\documentclass[oneside,english]{amsart}
\usepackage[T1]{fontenc}
\usepackage[latin9]{inputenc}
\usepackage{geometry}
\usepackage{amsthm}
\usepackage{graphicx}
\usepackage{amssymb}

\numberwithin{equation}{section} 
\numberwithin{figure}{section} 
\theoremstyle{plain}
\theoremstyle{plain}
\newtheorem{thm}{Theorem}
  \theoremstyle{plain}
  \newtheorem{prop}[thm]{Proposition}
  \theoremstyle{plain}
  \newtheorem*{prop*}{Proposition}
  \theoremstyle{remark}
  \newtheorem*{rem*}{Remark}
  \theoremstyle{plain}
  \newtheorem{cor}[thm]{Corollary}

\usepackage{babel}

\begin{document}
\global\long\def\R{\mathbb{R}}
\global\long\def\Q{\mathbb{Q}}
\global\long\def\Z{\mathbb{Z}}
\global\long\def\C{\mathbb{C}}
\global\long\def\P{\mathbb{P}}
\global\long\def\del{\partial}
\global\long\def\M{\mathcal{M}}
\global\long\def\F{\mathcal{F}}
\global\long\def\spinc{\mbox{Spin}^{c}}
\global\long\def\s{\mathfrak{s}}
\global\long\def\t{\mathfrak{t}}
\global\long\def\cptbar{\overline{\C\P}^{2}}
\global\long\def\cpt{\C\P^{2}}

\title{Nonorientable four-ball genus can be arbitrarily large}

\author{Joshua Batson}
\begin{abstract}
The nonorientable four-ball genus $\gamma_{4}(K)$ of a knot $K\subset S^{3}$
is the smallest first Betti number of any smoothly embedded, nonorientable
surface $F\subset B^{4}$ bounding $K$. In contrast to the orientable
four-ball genus, which is bounded below by the invariants $\sigma,\tau,$
and $s$, the best lower bound in the literature on $\gamma_{4}(K)$
for any $K$ is $3$. We prove that 

\[
\gamma_{4}(K)\geq\frac{\sigma(K)}{2}-d\left(S_{-1}^{3}(K)\right),\]
where the first term is half the knot signature, and the second is
the Heegaard-Floer $d$-invariant of the integer homology sphere given
by $-1$ surgery on $K$. In particular, we show that $\gamma_{4}(T_{2k,2k-1})=k-1$.
\end{abstract}
\maketitle

\section{Introduction}

One measure of the complexity of a knot $K\subset S^{3}$ is the complexity,
as codified by genus, of the simplest surface which bounds it. For
example, the only knot which bounds a genus zero surface embedded
in $S^{3}$ is the unknot. This definition of complexity depends dramatically
on the class of surfaces allowed: orientable or nonorientable, embedded
in $S^{3}$ or $B^{4}$, and for surfaces in $B^{4}$, whether or
not the embedding is smooth or locally flat. (The genus of a nonorientable
surface with boundary is defined to be its first Betti number $b_{1}$.)
For example, the nonalternating knot $11_{31}^{n}$ from Thistlethwaite's
table bounds an orientable surface of genus $3$ in $S^{3}$, a smooth
orientable surface of genus $2$ in $B^{4}$, and a locally flat orientable
surface of genus $1$ in $B^{4}$. Certifying the minimality of these
surfaces requires a variety of modern and classical knot invariants:
the Alexander polynomial $\Delta$ has degree $3$, Ozsvath-Szabo's
$\tau$-invariant is equal to $2$, and the Murasugi signature $\sigma$
is equal to $1$; to construct the final surface, Stominiew found
a genus one concordance to a knot with Alexander polynomial $1$,
which according to a result of Freedman bounds a locally flat disk.%
\footnote{http://www.indiana.edu/\textasciitilde{}knotinfo/ %
}

We know that orientable techniques cannot apply verbatim to obstruct
nonorientable surfaces because of a simple example: the $2k+1$-twist
torus knot $T_{2,2k+1}$ bounds a Mobius band in $S^{3}$, yet the
genus $k$ Seifert surface in Figure \ref{Flo:T_2k,2k+1} actually
has minimal genus even among orientable, locally flat surfaces embedded
in $B^{4}$ bounding the knot. The global property of orientability,
perhaps recast as the existence of a top homology class or a complex
structure, is somehow critical to both the proof and truth of the
bounds involving $\Delta$, $\tau$, and $\sigma$. While some obstructions
have been found to particular knots bounding Mobius bands or punctured
Klein bottles in $B^{4}$ (\cite{Yasuhara:1996tu,Murakami:2000wb}
see \cite{Gilmer:2011tf} especially for a comprehensive survey),
the following question remained open:

\textbf{Question. }Does every knot $K$ bound a punctured $\#^{3}\R\P^{2}$
smoothly embedded in $B^{4}$?%
\footnote{Gilmer and Livingston \cite{Gilmer:2011tf} use Casson-Gordon invariants
to construct a family of knots $K_{n}$ such that $K_{n}$ does not
bound a nonorientable \emph{ribbon} surface in $B^{4}$ of genus less
than $n$.%
}

The answer is, perhaps unsurprisingly, {}``no.''
\begin{thm}
\label{thm:main theorem}Suppose that $K\subset S^{3}$ bounds a smoothly
embedded, nonorientable surface $F\subset B^{4}$. Then\[
b_{1}(F)\geq\frac{\sigma(K)}{2}-d\left(S_{-1}^{3}(K)\right),\]
where $\sigma$ denotes the Murasugi signature and $d$ the Heegaard-Floer
$d$-invariant of the integer homology sphere given by $-1$ surgery
on $K$.
\end{thm}
In particular, we show
\begin{thm}
\label{thm:T-2k theorem}Any smoothly embedded surface $F\subset B^{4}$
bounding the torus knot $T_{2k,2k-1}$ has $b_{1}(F)\geq k-1$.
\end{thm}
The equivalent question in the topological category remains open.

\begin{figure}
\includegraphics[width=10cm]{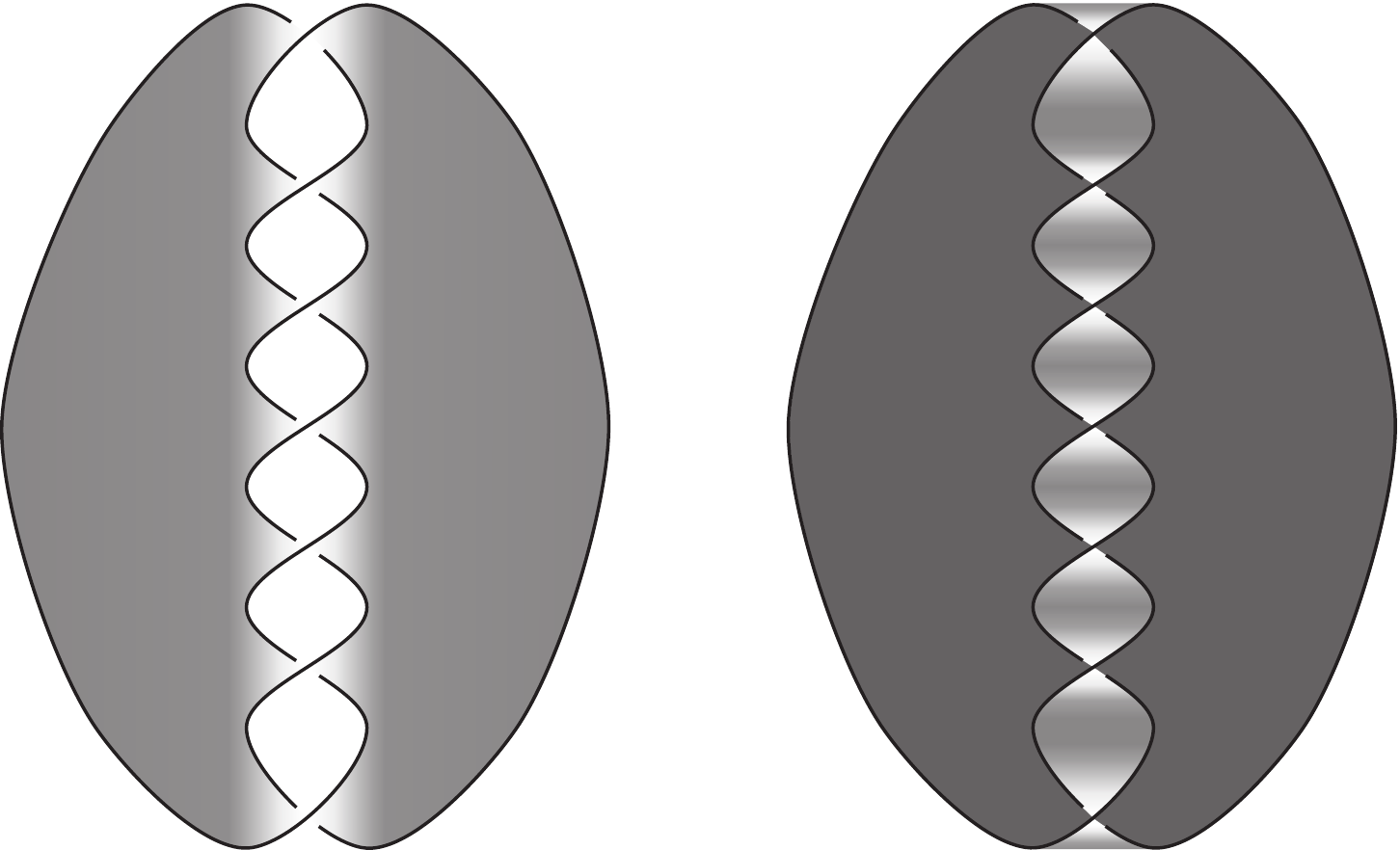}

\caption{An orientable surface with $b_{1}=k$ and a nonorientable surface
with $b_{1}=1$, each bounding $T_{2,2k+1}$. Here $k$ is $3$.}
\label{Flo:T_2k,2k+1}

\end{figure}

A moment for notation: the minimal genus of any surface bounding $K\subset S^{3}$
will be written $g_{3}(K)$, $g_{4}(K)$, $g_{4}^{top}(K)$, $\gamma_{3}(K)$,
$\gamma_{4}(K)$, $\gamma_{4}^{top}(K)$ depending on whether we allow
orientable or nonorientable surfaces ($g$ vs. $\gamma$) embedded
in $S^{3}$ or $B^{4}$ ($3$ vs. $4$) smoothly or topologically
(no superscript vs. $top$). Thus for $K=11_{31}^{n}$, we have \[
2g_{3}(K)=6\qquad2g(K)=4\qquad2g_{4}^{top}(K)=2\qquad\gamma_{3}(K)=3\qquad\gamma_{4}(K)=\;?\qquad\gamma_{4}^{top}(K)=\;?\]
This definite value for $\gamma_{3}(K)$, also called the \emph{crosscap
genus}, is due to Burton and Ozlen, who use integer programming and
normal surface theory to construct a triangulation of the knot complement
and certify minimal surfaces in it. Geometric techniques can also
be used to exactly compute the nonorientable 3-genus of a torus knot
$T_{p,q}$--it turns out that $\gamma_{3}(K)$, much like $\sigma(K)$,
is a recursive, arithmetic function of $p$ and $q$ \cite{Teragaito:2004te}.
Algebraic or polynomial invariants bounding $\gamma_{3}$ have yet
to be found--the 4-dimensional results of this paper provide the only
general 3-dimensional bounds known to the author.

The first large lower bounds on $g_{4}(K)$ are due to Murasugi, who
proved that $2g_{4}(K)\geq\sigma(K)$. The failure of this inequality
if we replace $2g$ with $\gamma$ is illuminating, so we quickly
recall a proof of the signature bound due to Gordon and Litherland.
Let $(F,\del F)\hookrightarrow(B^{4},S^{3})$ be an embedded surface
bounding a knot $K$. The normal bundle $\nu(F)$ always admits a
nonvanishing section $s$. On the boundary, $s\vert_{\del F}$ provides
a framing of $K$, which we use to define the \emph{normal Euler number}
of $F$:\[
e(F):=-\mbox{lk}(K,s(K)).\]
Gluing an orientable Seifert surface $\Sigma$ for $K$ to $F$ gives
a closed surface in $B^{4}$ with self-intersection $e(F)$. If $F$
is orientable, then $F\cup\Sigma$ represents an integral homology
class and self-intersection can be computed algebro-topologically;
since $H_{2}(B^{4};\Z)\cong0$, $e(F)$ must be zero. If $F$ is nonorientable,
then we must compute self-intersection geometrically. Take a transverse
pushoff of $F$, and choose arbitrary orientations in the neighborhood
of each intersection point. Together with the orientation of $B^{4}$,
this allows us to assign signs to each intersection; the sum turns
out to be independent of the choice of pushoff and local orientation.
It must be even, since we may compute self-intersection algebraically
over $\Z/2$, but it needn't be zero. For example, the Mobius band
bounding $T_{2,n}$ has normal Euler number $-2n$.

Let $W(F)$ denote the double cover of $B^{4}$ branched over $F$.
Gordon and Litherland \cite{Anonymous:2006wv} use the $G$-signature
theorem to show that the quantity \[
\sigma(W(F))+\frac{e(F)}{2}\]
is independent of the choice of surface $F$ bounding $K$, and equal
to the knot signature $\sigma(K)$. For any such $F$, then, \[
\left|\sigma(K)-\frac{e(F)}{2}\right|=\left|\sigma(W(F))\right|\leq b_{2}(W(F))=b_{1}(F),\]
where the final equality can be proved using elementary algebraic
topology.

This inequality is tight for both of the surfaces bounding $T_{2,2k+1}$
in Figure \ref{Flo:T_2k,2k+1}. The Seifert surface has $e(F)=0$
and $b_{1}(F)=2k$, the Mobius band has $e(F)=-2(2k+1)$ and $b_{1}(F)=1$,
and $\sigma(T_{2k,2k+1})=-2k$. In light of the important role played
by $e(F)$, it may be clarifying to sort surfaces based on the framing
they induce on the knot, and try to compute \[
\gamma_{4}(K,n):=\min\left\{ \gamma(F)\vert(F,\del F)\hookrightarrow(B^{4},K)\mbox{ and }e(F)=2n\right\} .\]
The signature inequality, in this notation, is $\gamma_{4}(K,n)\geq\left|\sigma(K)-n\right|$.

The strategy of this paper is as follows. First, we replace our nonorientable
surface in $B^{4}$ with an orientable surface in another manifold:
\begin{prop}
\label{pro:orientable replacement}Let $F\subset B^{4}$ be a smoothly
embedded nonorientable surface with odd $b_{1}$ bounding a knot $K\subset S^{3}$.
Then there exists an orientable surface $F^{\prime}\subset S^{2}\times S^{2}\backslash B^{4}$
which still bounds $K$, and has $b_{1}(F^{\prime})=b_{1}(F)-1$ and
$e(F^{\prime})=e(F)+2$.
\end{prop}
The construction is similar to one in \cite{Yasuhara:1996tu}.

We then attach a $-1$-framed $2$-handle along $K$ to get a four-manifold
$W$, with boundary $S_{-1}^{3}(K)$. There is a closed, orientable
surface $\Sigma$ in $W$, formed by union of $F^{\prime}$ and the
core of the $2$-handle. By excizing a neighborhood of $\Sigma$ from
$W$, we get a negative semi-definite cobordism from a circle bundle
over $\Sigma$ to $S_{-1}^{3}(K)$. The definiteness of $W$ gives
us an inequality between the Heegaard-Floer $d$-invariants of its
two boundaries, ultimately yielding:
\begin{thm}
\label{thm:d-invt theorem}Suppose that $K\subset S^{3}$ bounds a
smoothly embedded, nonorientable surface $F\subset B^{4}$. Then \[
\frac{e(F)}{2}\leq2d\left(S_{-1}^{3}(K)\right)+b_{1}(F).\]
That is, $\gamma_{4}(K,n)\geq n-2d(S_{-1}^{3}(K))$.
\end{thm}
Combining this theorem with the signature inequality yields Theorem
\ref{thm:main theorem}, which can be written as \[
\gamma_{4}(K)\geq\frac{\sigma(K)}{2}-d\left(S_{-1}^{3}(K)\right).\]

The $d$-invariants of integer homology spheres are in general somewhat
difficult to compute, though $d(S_{-1}^{3}(K))$ can in general be
calculated from the filtered Heegaard Floer knot complex $CFK^{\infty}(K)$
\cite{PETERS:ts}. When $K$ admits a lens space surgery, however,
these $d$-invariants can be read off from the Alexander polynomial
of $K$. Using a recursive formula of Murasugi to calculate the signature
of torus knots, we are able to prove Theorem \ref{thm:T-2k theorem},
that $\gamma_{4}(T_{2k,2k-1})\geq k-1$. In fact, we can construct
a surface $F_{2k,2k-1}$ bounding $T_{2k,2k-1}$ for which equality
holds.
\begin{prop}
The torus knot $T_{2k,2k-1}$ has $\gamma_{4}(T_{2k,2k-1})=k-1$.
That is, $T_{2k,2k-1}$ does not bound a punctured $\#^{k-2}\R\P^{2}$
smoothly embedded in $B^{4}$, and does bound a punctured $\#^{k-1}\R\P^{2}$.
\end{prop}
The surface $F_{2k,2k-1}$ is an example of a more general construction.
For each relatively prime $p$ and $q$, we find a nonorientable surface
$F_{p,q}$ in $B^{4}$ bounding $T_{p,q}$, whose first Betti number
satisfies the recursion $b_{1}(F_{p,q})=b_{1}(F_{p-2t,q-2h})+1$ where
$t$ and $h$ are the minimal nonnegative representatives of $-q^{-1}$
modulo $p$ and $p^{-1}$ modulo $q$, respectively. We conjecture
that these surfaces always have minimal genus, ie, that $b_{1}(F_{p,q})=\gamma_{4}(T_{p,q})$.

In contrast to the orientable case, where the so-called Milnor conjecture
$g_{4}(T_{p,q})=g_{3}(T_{p,q})$ holds, we show that $\gamma_{4}(T_{4,3})=1$
while Teregaito has computed that $\gamma_{3}(T_{4,3})=2$ \cite{Teragaito:2004te}.

\textbf{Acknowledgements. }The author would like to thank Peter Ozsváth
for many helpful conversations, and Jacob Rasmussen for comments on
a draft.

\section{Constructing an orientable replacement}

In this section, we prove
\begin{prop*}
[\bf\ref{pro:orientable replacement}]Let $F\subset B^{4}$ be a
smoothly embedded nonorientable surface with odd $b_{1}$ bounding
a knot $K\subset S^{3}$. Then there exists an orientable surface
$F^{\prime}\subset S^{2}\times S^{2}\backslash B^{4}$ which still
bounds $K$, and has $b_{1}(F^{\prime})=b_{1}(F)-1$ and $e(F^{\prime})=e(F)+2$.\end{prop*}
\begin{proof}
We break the proof into four steps.

\medskip{}

\emph{Step 1: There is an embedded disk $D\subset B^{4}$, with boundary
contained in $F$, such that $F\backslash\del D$ is orientable.}

\medskip{}

Since $F$ has odd $b_{1}$, it is diffeomorphic to a punctured orientable
surface boundary-connect summed with a Mobius band (Figure \ref{Flo:the surface F}).
Let $C\subset F$ be the core of the Mobius band; note that $F-C$
is orientable. After an ambient isotopy, we may arrange that $C$
lies in the sphere of radius $1/2$, $S_{1/2}^{3}\subset B^{4}$,
and that $F$ intersects $S_{1/2}^{3}$ transversly. Think of $C$
as a knot: it bounds some immersed disk $D^{2}$ in $S_{1/2}^{3}$,
with clasp and ribbon singularities (Figure \ref{Flo:clasp and ribbon singularities}).
We may remove the ribbon singularities by pushing the inner immersed
segment in towards the centre of the 4-ball. To remove the clasp singularities,
we push both immersed segments of the disk off the $1/2$-level, one
in towards the centre, and the other out towards the boundary. (The
ability to push the surface both inwards and outwards is crucial,
since a knot on the boundary of the $B^{4}$ bounds an embedded disk
only if it is slice.) By a small isotopy, we may arrange that this
embedded disk $D$ bounding $C$ intersects $F$ transversely on its
interior.

Let $N$ be a small regular neighborhood of $D$.

\begin{figure}
\includegraphics[width=7cm]{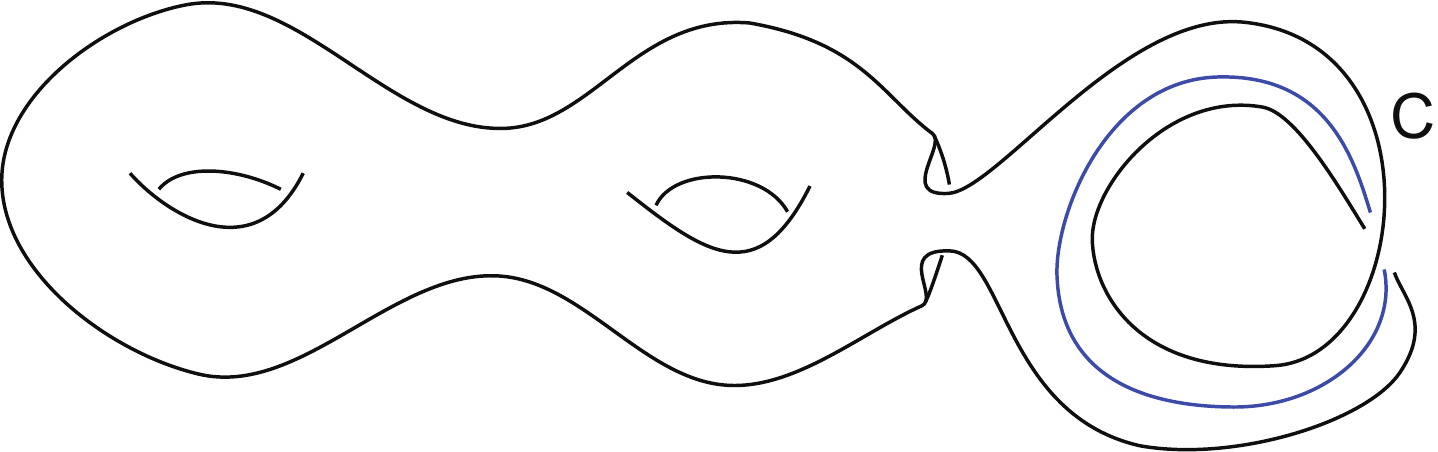}\caption{The surface $F$}
\label{Flo:the surface F}
\end{figure}
\begin{figure}
\includegraphics[width=8cm]{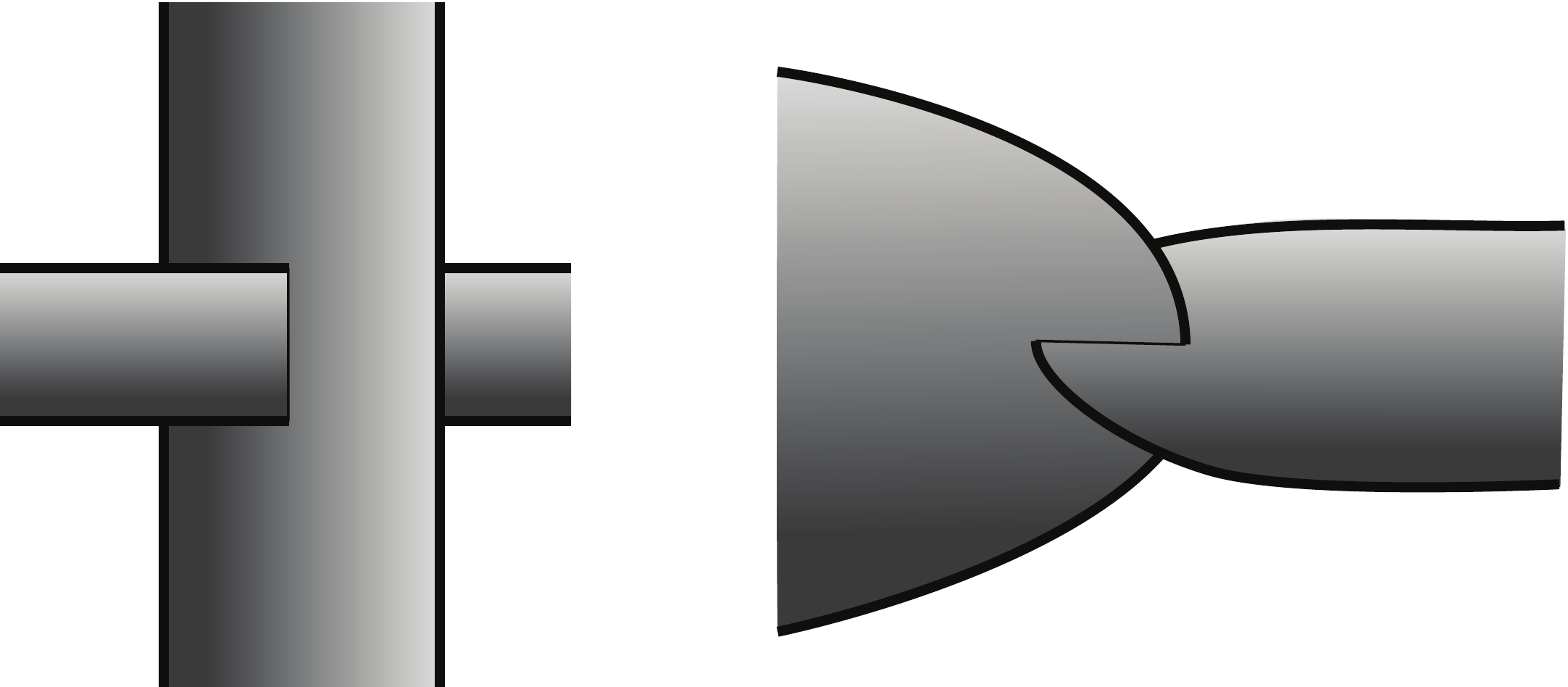}

\caption{Ribbon and Clasp singularities}
\label{Flo:clasp and ribbon singularities}
\end{figure}

\medskip{}

\emph{Step 2: The intersection $\del N\cap F$ is the link $L$ shown
in Figure \ref{Flo:F cap N}.}

\medskip{}

$N$ is diffeomorphic to $D\times D^{2}$, and intersects our surface
$F$ in a Mobius band (in the neighborhood of $\del D=C$) and a collection
of disks $pt\times D^{2}$ (neighborhoods of the transverse intersections
of $F$ with the interior of $D$). If we draw $S^{3}=\del N$ with
its standard decomposition into solid tori $S^{3}\cong S^{1}\times D^{2}\cup_{T^{2}}D^{2}\times S^{1}$,
we see $F\cap\del N$ as the link $L$ consisting of a $(2,2k+1)$-cable
of the core of the first factor, together with a collection of $l$
longitudes for the second. By construction, $L$ bounds a Mobius band
disjoint union a collection of $l$ disks in $N\cong B^{4}$.

\begin{figure}
\includegraphics[width=7cm]{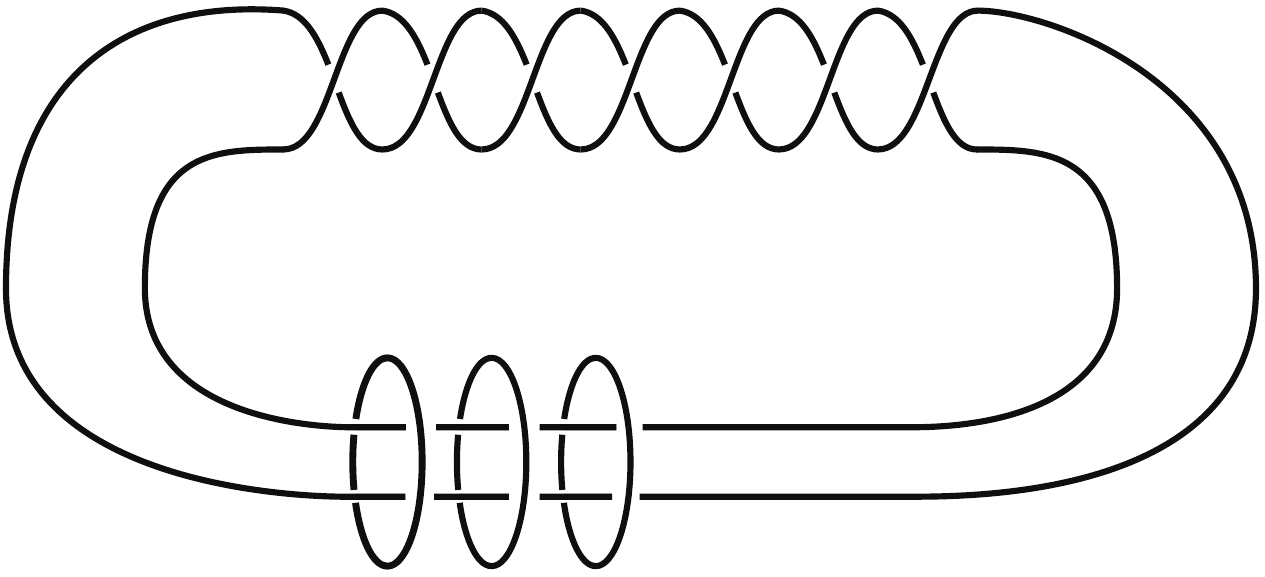}

\caption{Our surface $F$ intersects $\del N$ in a torus knot and an unlink}
\label{Flo:F cap N}
\end{figure}

\medskip{}

\emph{Step 3:} \emph{$L$ bounds $l+1$ disjoint embedded disks in
$S^{2}\times S^{2}\backslash B^{4}$}

\medskip{}

A handle decomposition for $S^{2}\times S^{2}\backslash B^{4}$ consists
of two zero-framed $2$-handles $H_{1}$ and $H_{2}$ attached along
a Hopf link in the boundary $S^{3}$, together with a $4$-handle.
To construct the slice disks for $L$, we begin with $\left|k\right|+l$
parallel copies of the core of $H_{2}$ and $2$ parallel copies of
the core of $H_{1}$--their boundaries form a multi-Hopf link, with
components $U_{1},\cdots,U_{\left|k\right|+l}$, $L_{1},L_{2}$, as
in the first frame of Figure \ref{Flo:building the new slice}. For
each $1\leq i\leq\left|k\right|$, connect $U_{i}$ to $L_{1}$ with
a twisted strip, and with one additional twisted strip, connect $V_{1}$
to $V_{2}$ . Call the surface consisting of the parallel cores and
the strips $E$, and note that the boundary of $E$ is isotopic to
$L$. Since each strip connects a distinct disk to $L_{1}$, $E$,
or rather a slightly isotoped copy of $E$, is a collection of $l+1$
disjoint embedded disks with boundary $L$.

\begin{figure}
\includegraphics[width=8cm]{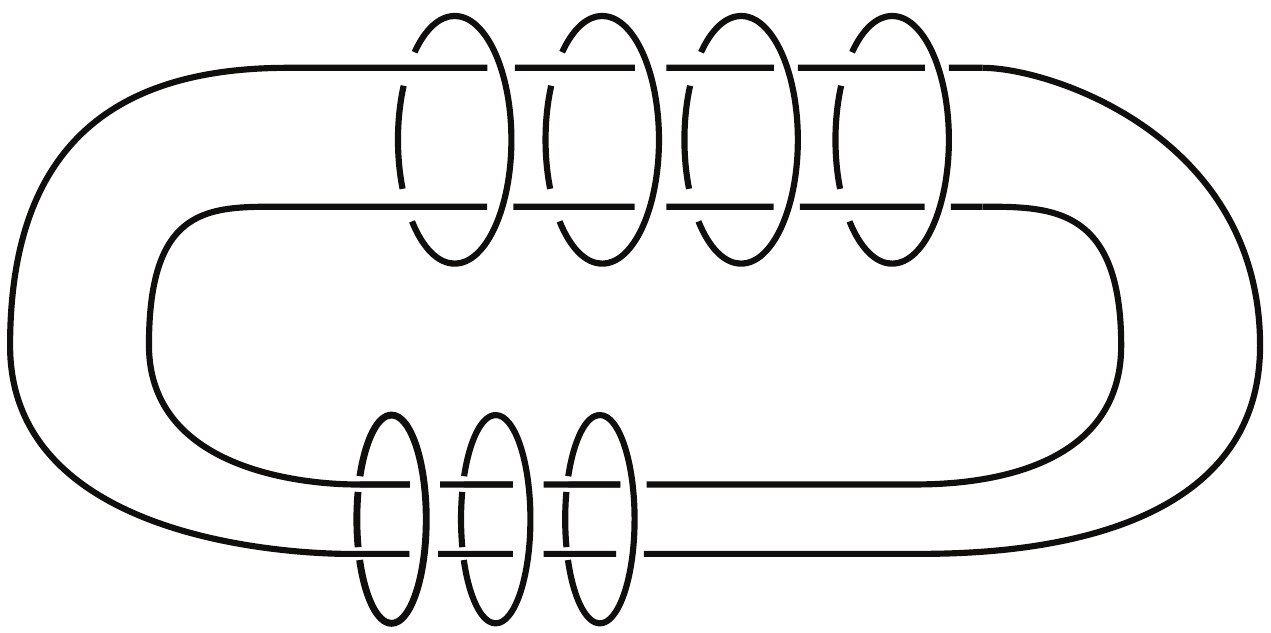}\includegraphics[width=8cm]{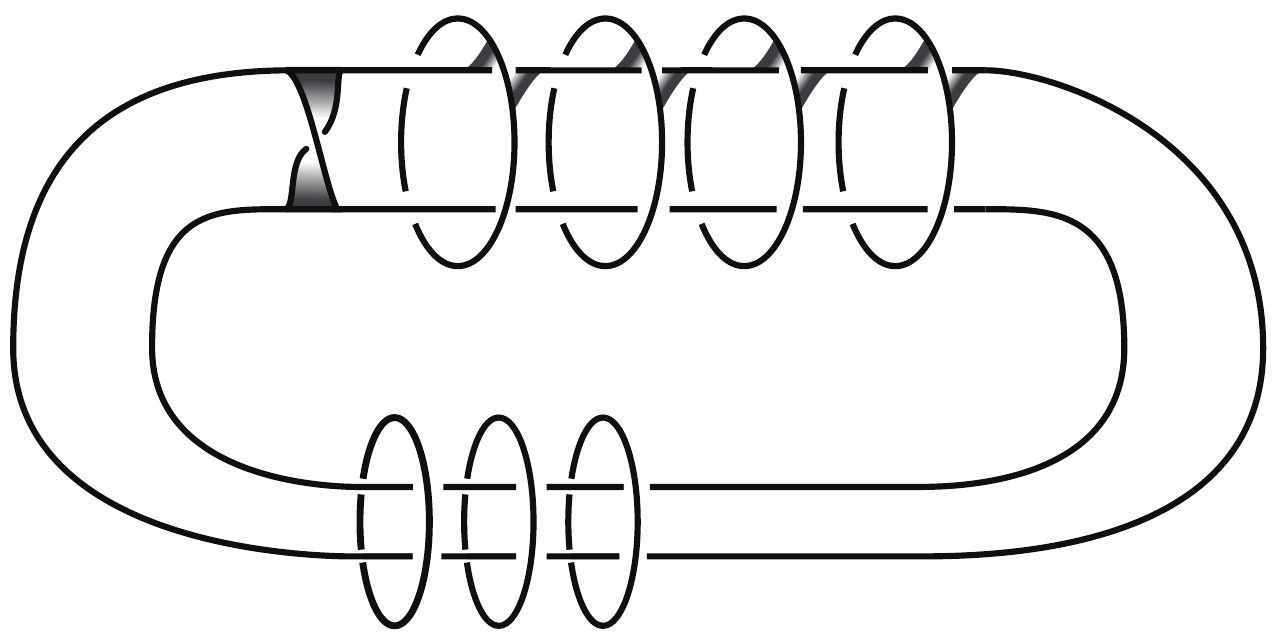}

\includegraphics[width=8cm]{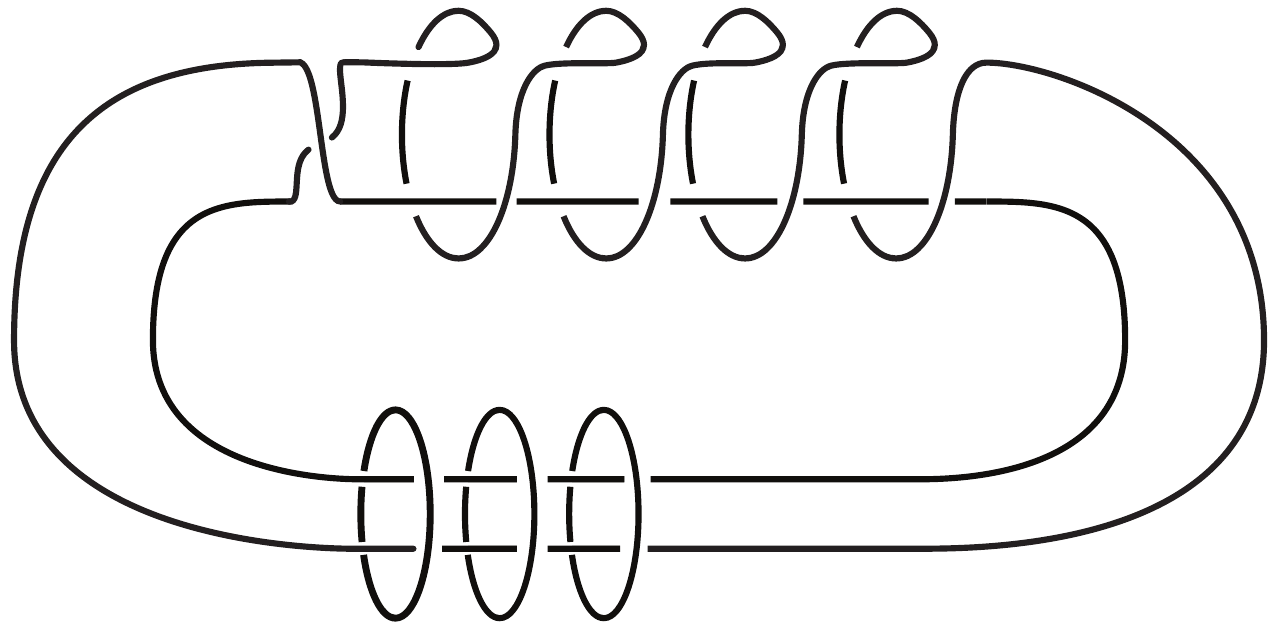}\includegraphics[width=8cm]{N_cap_F.pdf}\caption{For $k=-7$, $l=3$, we have drawn the multihopf link bounding a collection
of parallel disks, the strips which join them to form $E$, and the
boundary of $E$, which is isotopic to $L$.}

\label{Flo:building the new slice}
\end{figure}

\medskip{}

\emph{Step 4:} \emph{Construct $F^{\prime}$, and compute $b_{1}(F^{\prime})$
and $e(F^{\prime})$.}

\medskip{}

If we excize $N$ from $B^{4}$, we are left with an orientable surface
$F^{\prime\prime}\subset S^{3}\times[0,1]$, with boundary $K$ in
$S^{3}\times\{0\}$ and $L$ in $S^{3}\times\{1\}$. Attach $S^{2}\times S^{2}\backslash B^{4}$
along $S^{3}\times\{1\}$ to form a larger manifold, still diffeomorphic
to $S^{2}\times S^{2}\backslash B^{4}$. The slice disks $E$ for
$L$ combine with $F^{\prime\prime}$ to form an orientable surface
$F^{\prime}$, whose only remaining boundary is the original knot
$K$.

Since we have removed $l$ disks and an annulus from $F$, and replaced
them with $l+1$ disks, $b_{1}(F^{\prime})=b_{1}(F)-1$. It remains
to compare the normal Euler numbers. The remaining $l$ unknots, $U_{1},\dots,U_{l}$
have the same framing induced by $E$ and $F\cap N$. The torus knot
component of $L$ is bounded by a Mobius band in $F\cap N$, and by
an interesting disk in $E$. We invite the reader to verify that the
induced framings differ by $2$, due to the difference between the
vertical twisted strip connecting $V_{1}$ to $V_{2}$ in $E$ and
the horizontal one in Mobius band. That is, $e(E)=e(F\cap N)+2$.
Since Euler number, like any self-intersection, is additive, $e(F^{\prime})=e(F)+2$.
\end{proof}
For future reference, we note that the homology class $[F^{\prime}]\in H_{2}(S^{2}\times S^{2}\backslash B^{4})$
is $(2,m)$, in the basis given by $H_{1}$ and $H_{2}$, with $m=\left|k\right|+l$.
Since $F^{\prime}$ is orientable, its algebraic self-intersection
number, $4m$, must be equal to its geometric self-intersection number,
$e(F^{\prime})$.

\section{$d$-invariants}

Heegaard Floer homology associates to a $3$-manifold $Y$ equipped
with a $\spinc$ structure $\t$ a suite of $\Z[U]$-modules which
fit into a long exact sequence:\[
\cdots\rightarrow HF^{-}(Y,\t)\overset{i}{\rightarrow}HF^{\infty}(Y,\t)\overset{\pi}{\rightarrow}HF^{+}(Y,\t)\overset{\delta}{\rightarrow}HF^{-}(Y,\t)\rightarrow\cdots\]
If $c_{1}(\t)$ is torsion (in which case we also say that $\t$ is
torsion), then there is a $\Q$-grading $\mbox{gr}$ on the each of
these groups which is respected by $i$ and $\pi$. The action of
$U$ decreases grading by $2$. If $Y$ is a rational homology sphere,
then $HF^{\infty}(Y,\t)\cong\Z[U,U^{-1}]$, and every $\spinc$ structure
is torsion. In that case, the $d$-invariant (or \emph{correction
term}) $d(Y,\t)$ is the minimal grading of a non-$\Z$-torsion element
of $HF^{+}(Y,\t)$ in the image of $\pi$. 

If $b_{1}(Y)>0$, then there is an additional action of $H:=H_{1}(Y)/\mbox{Tors}$
on the $HF$ groups, which decreases grading by $1$. If for every
torsion $\t\in\spinc(Y)$, $HF^{\infty}(Y,\t)\cong\Z[U,U^{-1}]\otimes_{\Z}\Lambda^{*}H$,
then we say that $Y$ has standard $HF^{\infty}$. In that case, there
are many correction terms, one for each generator of $\Lambda^{*}H$.
We will be concerned with the bottom-most correction term, $d_{b}(Y,\t)$,
defined to be the minimal grading of a nontorsion element of $HF^{+}(Y,\t)$
in the image of $\pi$ \emph{and} in the kernel of the $H$-action.
The $d$-invariants terms will be useful to us because of their relationship
to definite cobordisms.
\begin{prop}
\cite{Ozsvath:2003ti} \label{pro:d-invt-inequality}Let $Y$ be a
closed oriented $3$-manifold (not necessarily connected) with standard
$HF^{\infty}$, endowed with a torsion $\spinc$ structure $\t$.
If $X$ is a negative semi-definite four-manifold bounding $Y$ such
that the restriction map $H^{1}(X;\Z)\rightarrow H^{1}(Y;\Z)$ is
trivial, and $\s$ is a $\spinc$ structure on $X$ restricting to
$\t$ on $Y$, then\[
c_{1}(\s)^{2}+b_{2}^{-}(X)\leq4d_{b}(Y,\t)+2b_{1}(Y).\]

\end{prop}
In the previous section, we constructed an orientable surface $F^{\prime}\subset S^{2}\times S^{2}\backslash B^{4}$
with boundary $K\subset S^{3}$. Attach a $-1$-framed $2$-handle
along $K$ to form a $4$-manifold $\overline{W}$ with boundary $S_{-1}^{3}(K)$
and intersection form\[
Q_{\overline{W}}=\left(\begin{array}{ccc}
-1 & 0 & 0\\
0 & 0 & 1\\
0 & 1 & 0\end{array}\right).\]
We may cap off $F^{\prime}$ with the core of the $2$-handle to form
a closed surface $\Sigma$ with genus $g=(b_{1}(F)-1)/2$, homology
class $(1,2,m)$, and self-intersection \[
n:=4m-1=e(F)+1>0\]
If we decompose $\overline{W}=\nu(\Sigma)\cup W$, then $W$ will
be a negative semi-definite cobordism from $Y_{g,n}$, the Euler number
$n$ circle bundle over $\Sigma$, to $S_{-1}^{3}(K)$. Alternatively,
we can view $W$ as a negative semi-definite four-manifold with disconnected
boundary $Y_{g,-n}\coprod S_{-1}^{3}(K)$. To apply the above proposition,
and so prove Theorem \ref{thm:d-invt theorem}, we need to understand
the homology, $HF^{\infty}$, and $d$-invariants of $Y_{g,n}$ and
$S_{-1}^{3}(K)$, and the intersection form on $W$.

The Gysin sequence for the disk bundle $\nu(\Sigma)$ gives

\[
0\rightarrow H^{1}(\nu(\Sigma))\rightarrow H^{1}(Y_{g,n})\rightarrow H^{2}(\nu(\Sigma),Y_{g,n})\overset{e}{\longrightarrow}H^{2}(\nu(\Sigma))\rightarrow H^{2}(Y_{g,n})\rightarrow H^{1}(\Sigma)\rightarrow0\]
where $e\in H^{2}(\nu(\Sigma))\cong\Z$ is $n$ times the generator.
Thus $H^{2}(Y_{g,n})\cong\Z^{2g}\oplus\Z/n$. Note that the restriction
of $H^{1}(\nu(\Sigma))$ to $H^{1}(Y_{g,n})$ is an isomorphism. Since
$H^{1}(\overline{W})=0$ (no $1$-handles were used in its construction),
the Mayer-Vietoris sequence

\[
0\rightarrow H^{1}(\overline{W})\rightarrow H^{1}(\nu(\Sigma))\oplus H^{1}(W)\rightarrow H^{1}(Y_{g,n})\rightarrow H^{2}(\overline{W})\rightarrow H^{2}(\nu(\Sigma))\oplus H^{2}(W)\rightarrow H^{2}(Y_{g,n})\]
shows that $H^{1}(W)=0$, trivially satisfying the restriction hypothesis
of Proposition \ref{pro:d-invt-inequality}. Since $H^{2}(\overline{W})\cong\Z^{3}$
has no $2$-torsion, a $\spinc$ structure on $\overline{W}$ is determined
by its first chern class. Any $\spinc$ structure on $W$ will give
us some inequality between $d$-invariants, but we will only need
to consider a certain $\spinc$ structure $\s_{t}$ with $PD(c_{1}(\s_{t}))=(\pm1,2,2a)$,
where \[
a=\frac{2(m-g)-1\pm1}{4}\]
and the sign is chosen so as to make $a$ an integer. The given vector
is characteristic for $Q_{\overline{W}}$, so does correspond to a
$\spinc$ structure. Crucially for our later use, $c_{1}(\s_{t})$
evaluates to $n-2g$ on $\Sigma$.

To compute the $c_{1}^{2}$ term in the proposition, we decompose
the intersection form of $\overline{W}$ in terms of the $\Q$-valued
intersection forms on $\nu(\Sigma)$ and $W$: if $c\in H^{2}(\overline{W})$,
then \[
Q_{W}(c)=Q_{\nu(\Sigma)}\left(c\vert_{\nu(\Sigma)}\right)+Q_{W}\left(c\vert_{W}\right).\]
A generator of $H^{2}(\nu(\Sigma),Y_{g,n})$ maps to $n$ times the
generator of $H^{2}(\nu(\Sigma))$ in the gysin sequence above, so
$Q_{\nu(\Sigma)}=\left(\frac{1}{n}\right)$. The value of $c\vert_{\nu(\Sigma)}\in H^{2}(\nu(\Sigma))$
is determined by integrating it over $\Sigma$, giving \begin{equation}
Q_{\overline{W}}(c)=\frac{\left\langle c,[\Sigma]\right\rangle ^{2}}{n}+Q_{W}\left(c\vert_{W}\right).\label{eq:intersectionform}\end{equation}
In our case, \[
c_{1}(\s_{t}\vert_{W})^{2}=Q_{\overline{W}}(c_{1}(\s_{t}))-\frac{\left\langle c_{1}(\s_{t}),[\Sigma]\right\rangle ^{2}}{n}=-1+8a-\frac{(n-2g)^{2}}{n}=-2\pm2-\frac{4g^{2}}{n}.\]

The relevant $d$-invariant of $Y_{g,-n}$ is computed in section
$9$ of \cite{Ozsvath:2003ti}, for use in their proof of the Thom
conjecture. If $n>2g$, then \[
d_{b}\left(Y_{g,-n},\s_{t}\vert_{Y_{g,-n}}\right)=\frac{1}{4}-\frac{g^{2}}{n}-\frac{n}{4}.\]
That calculation uses the integer surgeries exact sequence associated
to the Borromean knot in $K\subset\#^{2g}S^{1}\times S^{2}$: the
$-n$ surgery on $K$ gives $Y_{g,-n}$. Since $\#^{2g}S^{1}\times S^{2}$
has standard $HF^{\infty}$, so does $Y_{g,-n}$ (cf. Proposition
$9.4$ of \cite{Ozsvath:2003ti}). Finally, since $S_{-1}^{3}(K)$
is an integer homology sphere, it also has standard $HF^{\infty}$.

We are now ready to prove Theorem \ref{thm:d-invt theorem}. By Proposition
\ref{pro:d-invt-inequality}, we have

\[
c_{1}(\s_{t})^{2}+b_{2}^{-}(W)\leq4d_{b}(Y_{g,-n},\t)+4d(S_{-1}^{3}(K))+2b_{1}(Y_{g,-n})+2b_{1}(S_{-1}^{3}(K)).\]
After substituting all the values computed above, this reduces to\[
\left(-2\pm2-\frac{4g^{2}}{n}\right)+2\leq4\left(\frac{1}{4}-\frac{g^{2}}{n}-\frac{n}{4}\right)+4d\left(S_{-1}^{3}(K)\right)+2(2g).\]
If we take the unfavorable sign on $\pm2$, and recall that $b_{1}(F)=2g+1$
and $e(F)+1=n$, we get the inequality\begin{equation}
\frac{e(F)}{2}\leq2d\left(S_{-1}^{3}(K)\right)+b_{1}(F).\label{eq:d-invt ineq}\end{equation}
This argument relied on a value for $d_{b}(Y_{g,-n})$ only valid
if $n>2g$, ie, $e(F)+2\geq b_{1}(F)$. Proposition \ref{pro:d-invt-inequality},
applied to the surgery cobordism $S^{3}\rightarrow S_{-1}^{3}(K)$,
guarantees that $d(S_{-1}^{3}(K))\geq0$, so if $e(F)+2<b_{1}(F)$,
the above inequality is trivially satisfied. 

The initial construction of an orientable replacement required that
$b_{1}(F)$ be odd. Luckily, both sides of Equation \ref{eq:d-invt ineq}
change by the same amount under a positive real 'blow-up.' More precisely,
if we connect sum $F\subset B^{4}$ with the standard embedding of
$\R\P^{2}\subset S^{4}$ with Euler number $+2$, then both $b_{1}$
and $e/2$ increase by $1$. One way to construct this $\R\P^{2}$
is to glue together the Mobius band and disk bounding $T_{2,-1}$
(cf the mirror of Figure \ref{Flo:T_2k,2k+1} at $k=0$), then push
them off into opposite sides of $S^{3}\subset S^{4}$. If $b_{1}(F)$
is even, we may apply Equation \ref{eq:d-invt ineq} to $F\#\R\P^{2}$,
and so deduce it for $F$.

This completes the proof of Theorem \ref{thm:d-invt theorem}, and
hence of Theorem \ref{thm:main theorem}.
\begin{rem*}
Our final lower bound on $\gamma_{4}$ is the gap $\frac{\sigma(K)}{2}-d(S_{-1}^{3}(K))$.
For alternating knots, this quantity is nonpositive--in \cite{Ozsvath:2002wl},
Ozsváth and Szabó show that\[
d\left(S_{-1}^{3}(K)\right)=\max\left(0,2\left\lceil \frac{\sigma(K)}{4}\right\rceil \right)\]
For nonalternating knots, $\frac{\sigma(-)}{2}$ and $d(S_{-1}^{3}(-))$
can diverge widely, though both invariants satisfy a crossing-change
inequality \cite{PETERS:ts}:\[
\eta(K_{+})\leq\eta(K_{-})\leq\eta(K_{+})+2.\]
If $K$ becomes alternating after $c$ crossing changes, then $\frac{\sigma(K)}{2}-d(S_{-1}^{3}(K))$
can be as large as $2c$.
\end{rem*}

\section{Torus Knots}

Signatures of torus knots satisfy a recursion relation \cite{Murasugi:2008vv}.
If $\sigma(p,q):=\sigma(T_{-p,q})$, then\[
\sigma(p,q)=\begin{cases}
\sigma(q,p) & \mbox{ if }q>p\\
\sigma(p-2q,q)+q^{2}\;(-1) & \mbox{ if }2q<p\mbox{ }(q\mbox{ odd)}\\
-\sigma(2q-p,p)+q^{2}-2\;(+1) & \mbox{ if }2q>p\;(q\mbox{ odd})\\
p-1 & \mbox{ if }q=2\\
0 & \mbox{ if }q=1\end{cases}\]

Let $\sigma_{k}:=\sigma(T_{-2k,2k-1})=\sigma(2k,2k-1)$. Applying
the first and third conditions twice, we arrive at the recursion \[
\sigma_{k}=4k-2+\sigma_{k-1},\]
whence $\sigma_{k}=2k^{2}-2$.

The $d$-invariants of torus knots are also simple to compute, since
they admit lens space surgeries.
\begin{prop}
\cite{Ozsvath:2003ti} Let $K$ be a knot admitting a positive lens
space surgery. Then \[
d_{-1/2}(S_{0}^{3}(K))=-\frac{1}{2}\qquad\mbox{and}\qquad d_{1/2}(S_{0}^{3}(K))=\frac{1}{2}-2t_{0}\]
where if \[
\Delta_{K}(T)=a_{0}+\sum_{j=1}^{d}a_{j}(T^{j}+T^{-j})\]
then \[
t_{0}=\sum_{j=1}^{d}ja_{j}.\]

\end{prop}
The $d$-invariants of zero-surgery are related to those of $\pm1$-surgery
via Proposition 4.12 of \cite{Ozsvath:2003ti}:\[
d\left(S_{-1}^{3}(K)\right)=d_{-1/2}\left(S_{0}^{3}(K)\right)+\frac{1}{2}\qquad d\left(S_{1}^{3}(K)\right)=d_{1/2}\left(S_{0}^{3}(K)\right)-\frac{1}{2}.\]
Since $T_{p,q}$ admits a positive lens space surgery, we have\[
d\left(S_{-1}^{3}\left(T_{-p,q}\right)\right)=-d\left(S_{1}^{3}\left(T_{p,q}\right)\right)=-\left(d_{1/2}\left(S_{0}^{3}\left(T_{p,q}\right)\right)-\frac{1}{2}\right)=2t_{0}.\]
The Alexander polynomial of $T_{p,q}$ is \[
\Delta_{T_{p,q}}(T)=T^{-(p-1)(q-1)/2}\frac{(1-T)(1-T^{pq})}{(1-T^{p})(1-T^{q})}.\]

For torus knots $T_{2k,2k-1}$, the Alexander polynomial has a simple
form: \[
\Delta_{T_{2k,2k-1}}=\sum_{j=1}^{k-1}T^{j(2k-1)}-T^{j(2k-1)-(k-j)}+T^{-j(2k-1)}-T^{-j(2k-1)+(k-j)}\]
so \[
t_{0}=\sum_{j=1}^{k-1}j(2k-1)-(j(2k-1)-(k-j)=\sum_{j=1}^{k-1}k-j=\frac{k^{2}-k}{2}.\]
Hence \[
d\left(S_{-1}^{3}\left(T_{-2k,2k-1}\right)\right)=k^{2}-k.\]
The relevant difference between signature and $d$ is \[
\frac{\sigma}{2}-d=k^{2}-1-\left(k^{2}-k\right)=k-1.\]
Of course, the reflection of a surface bounding $T_{2k,2k-1}$ bounds
$T_{-2k,2k-1}$.

\begin{figure}
\includegraphics[width=0.7\columnwidth]{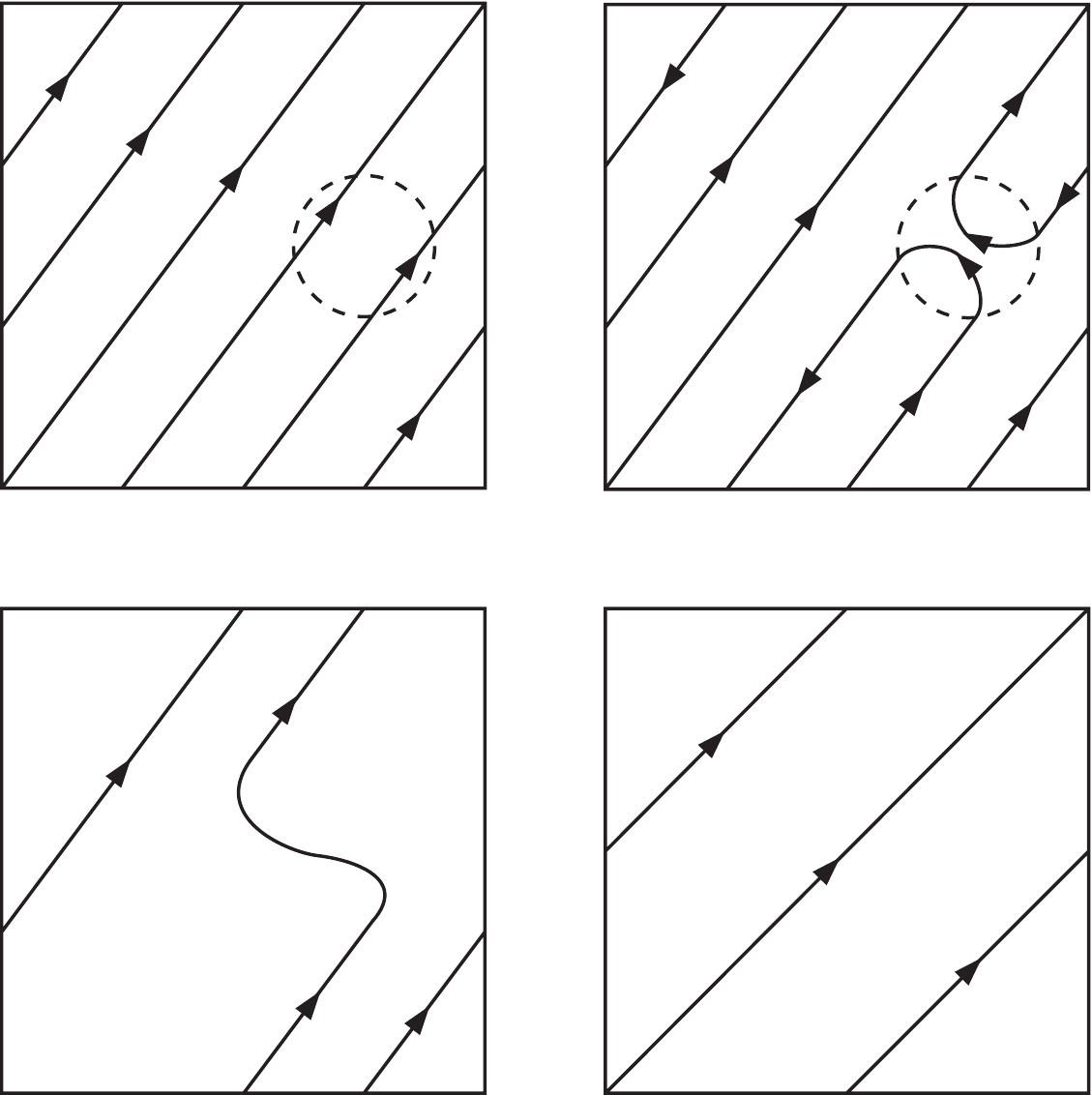}\caption{A cobordism from $T_{4,3}$ to $T_{2,1}$}
\label{Flo:torus cobordism}
\end{figure}

\begin{cor}
If $F\subset B^{4}$ is a smoothly embedded nonorientable surface
bounding $T_{2k,2k-1}\subset S^{3}$, then $b_{1}(F)\geq k-1$.
\end{cor}
We obtain this lower bound by the following construction. Consider
$T_{p,q}$ as actually lying in a standard torus, as in Figure \ref{Flo:torus cobordism}.
Take any two adjacent strands and join them with a strip, or, equivalently,
perform an index $1$ Morse move merging them. The resulting cobordism
is nonorientable, since the strands were parallel; it is a punctured
Mobius band. Since the resulting knot still lives on the torus, it
must be $T_{r,s}$ for some $r$ and $s$. The values of $r$ and
$s$ can be easily computed by orienting the resulting knot and counting
the signed intersection with the horizontal and vertical generators
of $H_{1}(T^{2})$. A short calculation shows that \[
r=p-2t\qquad s=q-2h\]
where $t\equiv-q^{-1}\mbox{ mod }p$, with $0\leq t<p$, and $h\equiv p^{-1}\mbox{ mod }q$,
with $0\leq h<q$. After an isotopy, $T_{r,s}$ will be in standard,
taut form on the torus, and we can repeat the process. Eventually,
we arrive at $T_{n,1}$ for some $n$, which is just an unknot. By
concatenating all of these cobordisms, then capping off the final
unknot with a disk, we have succesfully found a surface $F_{p,q}$
in $B^{4}$ bounding $T_{p,q}$.

For example, if $p=2k$ and $q=2k-1$, we have $t=-(-1)^{-1}=1$ and
$h=1^{-1}=1$, giving $r=2k-2$ and $s=2k-3$. Thus $T_{2k,2k-1}$
has a $\chi=-1$ cobordism to $T_{2(k-1),2(k-1)-1}$. Concatenate
$k-1$ of these, then cap off $T_{2,1}$ with a disk to get a closed
surface $F_{2k,2k-1}\subset B^{4}$ bounding $T_{2k,2k-1}$, with
$b_{1}(F_{2k,2k-1})=k-1$.

Since the isotopies and Morse moves take place inside of the torus,
we can actually embed each of these cobordisms in a thickened torus
$T^{2}\times[-\epsilon,\epsilon]$ in $S^{3}$, where we view the
$[-\epsilon,\epsilon]$ direction as a sort of time. The obstruction
to embedding all of $F_{p,q}$ in $S^{3}$ is that the final disk
bounding $T_{n,1}$ cuts through all of the previous layers, unless
$n=0$. To get a surface in $S^{3}$, we must continue with these
within-torus cobordisms: $T_{n,1}\mapsto T_{n-2,1}\mapsto\cdots$.
If $n$ is even, or, equivalently, if $pq$ was even to start, then
we do get a surface in $S^{3}$. Teragaito has computed $\gamma_{3}(T_{p,q})$,
and it agrees with $b_{1}(F)$ \cite{Teragaito:2004te}. For example,
$\gamma_{3}(T_{2k,2k-1})=k$. If $n$ is odd, then this construction
fails to give a surface in $S^{3}$, though a slight modification
(cf. \cite{Teragaito:2004te} Remark 4.9) will do.

We conjecture that the surfaces $F_{p,q}$ bounding $T_{p,q}$ are
best possible, that $b_{1}(F_{p,q})=\gamma_{4}(T_{p,q})$. Many pairs
$(p,q)$ for which the conjecture holds can be certified using the
$d$-invariant bounds of this paper. Similar invariants, derived by
considering larger surgeries on the knot, give even more examples.
These stronger bounds will be discussed in a forthcoming paper.

\bibliographystyle{math.bst}
\bibliography{nonorientable.bib}

\end{document}